\documentclass[12pt]{article}
\usepackage[dvips]{graphicx}%
\usepackage{amsmath, amssymb,amsthm, amsfonts, mathrsfs}%
\usepackage[matrix,arrow,curve]{xy}

\newtheorem{theorem}{Theorem}
\theoremstyle{definition}

\newtheorem*{remark}{Remark}

\newcommand{\tr}{\mathop{\mathrm{tr}}}

\begin{document}

\begin{center}
\Large \textsf{Invariants of symbols of the linear differential operators}\footnote{The work is partially supported by the Russian Foundation for Basic Research, Grant 18-29-10013.}

\vspace{5pt}

\normalsize \textsf{Pavel Bibikov\footnote{e-mail: tsdtp4u@proc.ru},
Valentin Lychagin\footnote{e-mail: valentin.lychagin@uit.no}}

\vspace{5pt}

Institute of Control Sciences RAS, 117997 Moscow
(Russia), 
\end{center}

\abstract{In this paper we classify the
symbols of the linear differential operators of order $k$, which act from the module $C^\infty(\xi)$ to the module $C^\infty(\xi^t)$, where $\xi\colon E(\xi)\to M$ is vector bundle over the smooth manifold $M$, bundle $\xi^t$ is either $\xi^*$ with fiber $E^*:=\mathrm{Hom}(E,\mathbb{C})$ or $\xi^\flat$ with fiber $E^\flat:=\mathrm{Hom}(E, \Lambda^n T^*)$ and $C^\infty(\xi)$, $C^\infty(\xi^t)$ are the modules of their smooth sections. 
To find invariants of the symbols we associate with every non-degenerated symbol the tuple of linear operators acting on space $E$ and reduce our problem to the classification of such tuples with respect to some orthogonal transformations. Using the results of C. Procesi, we find generators for the
field of rational invariants of the symbols and in terms of these invariants provide a criterion of equivalence of non-degenerated
symbols.}

\vspace{10pt}

2000 MSC: 22E26, 32L05, 53A55%

Keywords: linear differential operator, symbol, vector bundle, homogeneous form, differential invariant, jet space%

Subject classification: real and complex differential geometry, jet
spaces, differential equations.


\section{Introduction}

This paper is a continuation of the work~\cite{LY4}, where the equivalence problem of the linear differential operators acting on the smooth sections of the vector bundles, was studied. 

The study of invariants of the linear differential operators is important and has a lot of applications in various areas of mathematics. The first work of this area belongs to Bernard Riemann, who discovered the curvature of the second order scalar differential operator, defined by its symbol. 

In the present paper we discuss the problem of classification for 
symbols of linear differential operators, acting in vector bundles.

Namely, we consider two vector bundles $\xi\colon E(\xi)\to M$ and $\xi^t\colon E^t(\xi^t)\to M$ over the smooth manifold $M$ of dimension $n$ with the modules $C^\infty(\xi)$ and $C^\infty(\xi^t)$ of their smooth sections (we provide all constructions over the complex field $\mathbb{C}$ and the real field $\mathbb{R}$) and the modules $\mathbf{Diff}_k(\xi, \xi^t)$ of the linear differential operators of order $k$. Here vector bundle $\xi^t$ is one of two possible bundles:
\begin{itemize}
    \item $\xi^t=\xi^*\colon \mathrm{Hom}(\xi, \mathbb{R})\to M$
    \item $\xi^t=\xi^\flat \colon \mathrm{Hom}(\xi, \Lambda^n\tau^*)\to M$, where $\Lambda^n\tau^*$ is the bundle of differential $n$-forms on manifold $M$.
\end{itemize}

Each operator $\Delta\in \mathbf{Diff}_k(\xi,\xi^t)$ acts from the first module to the second one, i.e. $\Delta\colon C^\infty(\xi)\to C^\infty(\xi^t)$. Note that the classification of the linear differential operators itself (see~\cite{LY4}) consists of two parts. In the first, ``algebraic'', part the invariants of the symbols of operators are constructed, and in the second, ``geometric'', part the canonical connection, associated with the differential operator, is defined, and using these connections and the procedure of quantazation the invariants of symbols are prolonged to invariants of the operators. The aim of this paper is the realization of the first, ``algebraic'' part of this scheme, in general situation of arbitrary $k$-th order symbols of differential operators $\Delta\in \mathbf{Diff}_k(\xi,\xi^t)$.

Let us fix a point $a\in M$ and consider the value of the symbol of differential operator at this point. Then the symbol of a linear differential operator $\Delta\in \mathbf{Diff}_k(\xi,\xi^t)$ becomes a tensor $\sigma(\Delta)$, which could be viewed as a fiber-wise homogeneous $k$-form on the cotangent space $T^*=T_a^*M$ with values in the space $\mathrm{Hom}(E,E^t)$, where $E=\xi^{-1}(a)$ and $E^t=(\xi^t)^{-1}(a)$ are the fibers of bundles $\xi$ and $\xi^t$ correspondingly. The group of automorphisms of bundle $\xi$, that preserve the point $a$, acts on the space of such symbols in the following way: on space $T^*$ it acts by linear transformations (this action naturally prolongs to the action on the space $\mathrm{\mathsf{S}}^kT^*$ on the space of homogeneous $k$-forms), and on $\mathrm{Hom}(E,E^t)$ it acts by left-right mutliplications, i.e. we consider the group $\mathrm{GL}(T)\times\mathrm{GL}(E)$. Note that the action of group $\mathrm{GL}(E)$ has non-trivial kernel $\{\pm \mathbf{1}\}\simeq \mathbb{Z}_2$. So we study the action of group $G:=\mathrm{GL}(T)\times (\mathrm{GL}(E)/\mathbb{Z}_2)$.

To find $G$-invariants of the symbols $\sigma\colon \mathrm{\mathsf{S}}^kT^*\to \mathrm{Hom}(E,E^t)$, we use the special case of the so-called first fundamential theorem in classical invariant theory, which describes the invariants of the linear action of the classical groups on the system of linear operators acting on some vector space. In our situation we need the special case of this theorem, which was proved by C. Procesi (see~\cite{Pr1}). This result of Procesi works with the actions of orthogonal and symplectic groups on the tuples of linear operators and describes the generators of corresponding invariant field. We have to consider the field of rational invariants instead of traditional algebra of polynomial invariants, because the action of the center $\{\lambda \cdot\mathbf{1}\}\subset \mathrm{GL}(E)$ on space $\mathrm{Hom}(E,E^t)$ is non-trivial, so there are no polynomial invariants in our problem. 

Using the Procesi invariants for the tuples of linear operators, we construct the set of rational invariants, which separate the regular $G$-orbits of the symbols and hence, according to Rosenlicht theorem, generate the field of rational invariants. It is interesting to note that our invariants are defined not on the space $T^*$, and even not on the space $\mathrm{\mathsf{S}}^kT^*$, but on the space $\big(\mathrm{\mathsf{S}}^kT^*\big)^{\times N}$, where $N=\binom{n+k-1}{k}$. 

\section{Notations and definitions}

In this section we provide necessary definitions and notations for linear differential operators acting on the smooth sections of the vector bundles.

\subsection{Basic definitions}

Let $\xi\colon E(\xi)\to M$ be vector bundle of rank $m$ over the smooth $n$-dimensional manifold $M$. Denote by $\tau\colon TM\to M$ and $\tau^*\colon T^*M\to M$ the tangent and cotangent bundles over manifold $M$. Their symmetric powers we denote by $\mathrm{\textsf{S}}^k \tau$ and $\mathrm{\textsf{S}}^k \tau^*$ respectively. Also consider the trivial bundle $\mathbf{1}\colon \mathbb{R}\times M\to M$.

Modules of the smooth sections of bundles $\xi$, $\mathrm{\textsf{S}}^k \tau$, $\mathrm{\textsf{S}}^k \tau^*$ and $\mathbf{1}$ will be denoted as $C^\infty(\xi)$, $\Sigma_k(M):=C^\infty(\mathrm{\textsf{S}}^k \tau)$, $\Sigma^k(M):=C^\infty(\mathrm{\textsf{S}}^k \tau^*)$ and $C^\infty(M)$ respectively.

Now we take the two vector bundles $\xi$, $\xi^t$ and the corresponding modules. The module of linear differential operators of order $\leq k$ will be denoted by $\mathbf{Diff}_k(\xi,\xi^t)$. 

Factor-module $\mathbf{Smbl}_k(\xi,\xi^t):=\mathbf{Diff}_k(\xi,\xi^t)/\mathbf{Diff}_{k-1}(\xi,\xi^t)$ is called \emph{the symbol module}, and the image $\sigma_k(\Delta)\in \mathbf{Smbl}_k(\xi,\xi^t)$ of the linear differential operator $\Delta\in\mathbf{Diff}_k(\xi,\xi^t)$ after the canonical projection $\mathrm{smbl}\colon \mathbf{Diff}_k(\xi,\xi^t)\to \mathbf{Smbl}_k(\xi,\xi^t)$ is called \emph{the symbol of $\Delta$}.

It is also known, that $\mathbf{Smbl}_k(\xi,\xi^t)\simeq \mathrm{Hom}(\xi,\xi^t)\otimes \Sigma_k(M)$.

From now on we fix a point $a\in M$ and consider the restriction values of symbols at this point. So by a symbol we mean a tensor $\sigma\in \mathrm{Hom}(E,E^t)\otimes \mathrm{\textsf{S}}^k T$, or, in another terms, the linear map $\sigma\colon \mathrm{\textsf{S}}^k T^*\to \mathrm{Hom}(E,E^t)$. Here $E=\xi^{-1}(a)$ and $E^t=(\xi^t)^{-1}(a)$ are the fibers of bundles $\xi$ and $\xi^t$ over point $a$ and $T=\tau^{-1}(a)$ is the tangent plane to manifold $M$ at point $a$.

Group $G:=\mathrm{GL}(T)\times (\mathrm{GL}(E)/\mathbb{Z}_2)$ acts on the module $\mathbf{Smbl}(\xi,\xi^t)$ by linear
transformations: the first part $\mathrm{GL}(T)$ acts by linear
transformations on the base $T$, which prolong to the linear transformations of the symmetric power $\mathrm{\textsf{S}}^kT^*$, and the second part
$\mathrm{GL}(E)/\mathbb{Z}_2$ acts by left and right shifts on the fibers $E$ and $E^t$: if $\sigma_q:=\sigma(q)\in \mathrm{Hom}(E,E^t)$ is the image of homogeneous $k$-form $q\in\mathrm{\textsf{S}}^kT^*$, then the action is$$A\circ \sigma_q = A^{-t}\sigma_q A^{-1},$$ where $A\in \mathrm{GL}(E)/\mathbb{Z}_2$, $t$ is the dual operation and $A^{-t}:=(A^{-1})^t$. We classify the symbols $\sigma\colon \mathrm{\textsf{S}}^k T^*\to \mathrm{Hom}(E,E^t)$ with respect to this action.

\section{Invariants}

In this section we provide the construction of the rational invariants for the action of our group $G$ on the space $\mathbf{Smbl}(\xi,\xi^t)$ of symbols. 

We consider the following cases:
\begin{itemize}
    \item $E^t=E^*:=\mathrm{Hom}(E, \mathbb{R})$
    \item $E^t=E^\flat:=\mathrm{Hom}(E, \Lambda^nT^*)$
\end{itemize}

These two cases are based on the same construction. Namely, we reduce our problem to classification of the tuples of linear operators with respect to conjugations by orthogonal group. The following result belongs to C. Procesi (see~\cite{Pr1}).

\begin{theorem}\label{th.Pr}
Let $V$ be a Euclidian vector space with symmetric form $g$, $\mathcal{X}_1$, \ldots, $\mathcal{X}_\ell\in\mathrm{Hom}(V)$ be the tuple of linear operators on $V$, and group $\mathrm{O}_g(V)$ acts on such tuples by conjugations. Then the algebra of polynomial invariants for such action is generated by the polynomials $\tr (\mathcal{Y}_{i_1}\ldots \mathcal{Y}_{i_s})$, associated with non-commutative monomials $\mathcal{Y}_{i_1}\ldots \mathcal{Y}_{i_s}$ of the length $s\leq 2^{\dim V}-1$, where $\mathcal{Y}_{i_j}=\mathcal{X}_{i_j}$ or $(\mathcal{X}_{i_j})_g$ (the dual operator with respect to the symmetric bi-linear form $g$).
\end{theorem}

The similar result holds also for the symplectic group.

\begin{theorem}\label{th.Pr.S}
Let $V$ be a symplectic vector space with skew-symmetric form $\omega$, $\mathcal{X}_1$, \ldots, $\mathcal{X}_\ell\in\mathrm{Hom}(V)$ be the tuple of linear operators on $V$, and group $\mathrm{Sp}_\omega(V)$ acts on such tuples by conjugations. Then the algebra of polynomial invariants for such action is generated by the polynomials $\tr (\mathcal{Y}_{i_1}\ldots \mathcal{Y}_{i_s})$, associated with non-commutative monomials $\mathcal{Y}_{i_1}\ldots \mathcal{Y}_{i_s}$ of the length $s\leq 2^{\dim V}-1$, where $\mathcal{Y}_{i_j}=\mathcal{X}_{i_j}$ or $(\mathcal{X}_{i_j})_\omega$ (the dual operator with respect to the skew-symmetric bi-linear form $\omega$).
\end{theorem}

We apply the Procesi theorems for the special tuples of linear operators associated with the symbols $\sigma\colon \mathrm{\textsf{S}}^k T^*\to \mathrm{Hom}(E,E^t)$ and get the set of rational functions on copies of $\mathrm{\textsf{S}}^k T^*$. Thus we ``kill'' the action of the group $\mathrm{GL}(E)/\mathbb{Z}_2$. 

\subsection{Preparation}

Let $E$ and $T$ be two vector spaces and $E^t$ is one of two dual spaces: either $E^*=\mathrm{Hom}(E,\mathbb{C})$ or $E^\flat=\mathrm{Hom}(E, \Lambda^n T^*)$, where $n=\dim T$ and $\Lambda^n T^*$ is the space of exterior $n$-forms. 

For $x\in E$ and $u\in E^t$ denote by $\langle x,u\rangle = \langle u,x\rangle$ the paring of the elements $x\in E$ and $u\in E^t$, so in case $E^t=E^*$ one has $\langle\cdot,\cdot\rangle\colon E\times E^*\to \mathbb{R}$, and in case $E^t=E^\flat$ one has $\langle\cdot,\cdot\rangle\colon E\times E^\flat\to \Lambda^n T^*$.

Let $b\in \mathrm{Hom}(E,E^t)$ be a linear operator. Then the dual operator $b^t\in \mathrm{Hom}(E,E^t)$ is defined be relation $\langle bx,y\rangle = \langle x,b^ty\rangle$, for all $x$, $y\in E$. 

Also we identify this map with bi-linear form $\langle bx,y\rangle$, $x$, $y\in E$, with the values in $\mathbb{C}$ or $\mathbb{R}$ (when $E^t=E^*$) or in $\Lambda^n T^*$ (when $E^t=E^\flat$). 

We denote this form also by $b$, so $b(x,y)=\langle bx, y\rangle$ and $b(y,x)=\langle b^tx, y\rangle$ for all $x$, $y\in E$.

Let $b\in\mathrm{Hom}(E, E^t)$ be a reversible bi-linear form and $\mathcal{A}\in\mathrm{End}(E)$ be the linear operator. We define the $b$-dual linear operator $\mathcal{A}_b$ as follows: $b(\mathcal{A}x,y)=b(x,\mathcal{A}_by)$ for all $x$, $y\in E$. 

Note, that $$\langle b\mathcal{A}x, y\rangle=b(\mathcal{A}x,y)=b(x,\mathcal{A}_by)=\langle bx, \mathcal{A}_by\rangle=\langle (\mathcal{A}_b)^tbx, y\rangle,$$ so $$\mathcal{A}_b=(b\mathcal{A}b^{-1})^t=b^{-t} \mathcal{A}^t b^t.$$

The bi-linear forms $b_s:=\frac{b+b^t}{2}$ and $b_a:=\frac{b-b^t}{2}$ give us the splitting of bi-linear form $b$ in the sum of symmetric and skew-symmetric parts.

Assume, that bi-linear form $b$ and its symmetric part $b_s$ are reversible. Then we associate with  bi-linear form $b$ two linear operators $$\mathcal{H}^b, \;\mathcal{S}^b\in \mathrm{End}(E), \quad \mathcal{H}^b:=b^{-t} b,\quad\text{and}\quad \mathcal{S}^b:=b_s^{-1}b_a.$$
There is the following relation between these two operators:
$$\mathcal{S}^b = (b+b^t)^{-1}(b-b^t) = ((b^{-1})^t b+1)^{-1}(b^{-1})^t b^t((b^{-1})^tb-1) = (\mathcal{H}^b+1)^{-1}(\mathcal{H}^b-1).$$ Moreover, $$\mathcal{H}^b=(1+\mathcal{S}^b)(1-\mathcal{S}^b)^{-1}$$ and
\begin{align*}
\mathcal{S}^b_{b_s} &= b_s^{-1}(\mathcal{S}^b)^t b_s = b_s^{-1}(-b_a)(b_s^{-1})^tb_s = -\mathcal{S}^b, \\ (\mathcal{H}^b)^{-t} b_{s,a} (\mathcal{H}^b)^{-1} &=
b(b^{-1})^t\frac{b\pm b^t}{2} b^{-1}b^t = \frac{b\pm b^t}{2}=b_{s,a}. \end{align*}
Therefore, operator $\mathcal{H}^b$ preserves both forms $b_s$ and $b_a$, $$\mathcal{H}^b\in \mathrm{O}_{b_s}(E)\cap \mathrm{Sp}_{b_a}(E).$$

We have the following actions of the group $\mathrm{GL}(E)$ on the space $\mathrm{End}(E)$ of linear operators and on the space $\mathrm{Hom}(E,E^t)$ of bi-linear forms by the following relations:
\begin{equation}\label{eq.E}
A\circ \mathcal{X} = A\mathcal{X}A^{-1}\quad\text{and}\quad A\circ b=A^{-t} bA^{-1},
\end{equation}
where $A\in \mathrm{GL}(E)$, $\mathcal{X}\in\mathrm{End}(E)$ and $b\in\mathrm{Hom}(E,E^t)$ is a bi-linear form. 

Then, $$\mathcal{H}^{A\circ b} = (A\circ b)^{-t} (A\circ b) = (Ab^{-t} A^t)(A^{-t}bA^{-1}=A(b^{-t}b)A^{-1}=A\circ \mathcal{H}^b.$$ In similar way we get $$\mathcal{S}^{A\circ b} = A\circ\mathcal{S}^b.$$

We call bi-linear form $b$ non-degenerated, if $b$ and $b_s$ are reversible and operator spectrum of operator $\mathcal{S}^b$ is simple (i.e. all eigenvalues of $\mathcal{S}^b$ are distinct).

The following theorem is valid.

\begin{theorem}\label{th.b-A}
1. Two non-degenerated bi-linear forms $b$ and $\widetilde{b}$ are $\mathrm{GL}(E)$-equivalent if and only if symmetric forms $b_s$ and $\widetilde{b}_s$ are $\mathrm{GL}(E)$-equivalent and operators $\mathcal{S}^b$ and $\mathcal{S}^{\widetilde{b}}$ are $\mathrm{GL}(E)$-equivalent.

2. Field of rational $\mathrm{GL}(E)$-invariants of bi-linear forms is generated by functions $\tr (\mathcal{S}^b)^k$, where $k\leq \dim E$ is even.
\end{theorem}

\begin{proof}
The first part of the theorem immediately follows from the Prochesi theorem~\ref{th.Pr} and the relation $b_a=b_s\mathcal{S}^b$. To prove the second part note that the codimension of the general $\mathrm{GL}(E)/\mathbb{Z}_2$-orbit of bi-linear form equals $\Big[\frac 1 2\dim E\Big]$, so the field of rational $\mathrm{GL}(E)/\mathbb{Z}_2$-invariants of bi-linear forms has the transcendence degree $\Big[\frac 1 2\dim E\Big]$. On the other hand functions $\tr (\mathcal{S}^b)^{2k}$, where $k=1$, \ldots, $\Big[\frac{\dim E}{2}\Big]$ are $\mathrm{GL}(E)/\mathbb{Z}_2$-invariant and algebraically independent. Hence, according to Rosenlicht theorem (see~\cite{Rosenlicht}), these functions generate the field of rational $\mathrm{GL}(E)/\mathbb{Z}_2$-invariants.
\end{proof}

\begin{remark}
For the case of the real field $\mathbb{R}$ there is one more discrete $\mathrm{GL}(E)$-invariant of bi-linear form: signature of the symmetric part $b_s$.
\end{remark}

\subsection{Generators of the field of rational $G$-invariants}

Here we apply the results and constructions from the previous section for the classification of the symbols $\sigma\colon \mathrm{\mathsf{S}}^kT^*\to \mathrm{Hom}(E,E^t)$ over the complex field $\mathbb{C}$.

Let $q_1$, \ldots, $q_N\in \mathrm{\mathsf{S}}^kT^*$ be an arbitrary homogeneous $k$-forms. Then their images $\sigma_{q_i}:=\sigma(q_i)$ are bi-linear forms from $\mathrm{Hom}(E,E^t)$. The group $\mathrm{GL}(E)/\mathbb{Z}_2$ acts on such bi-linear forms by formula~(\ref{eq.E}). 

Assume that all bi-linear forms $\sigma_{q_i}$ are non-degenerated. Put $g:=(\sigma_{q_1})_s$ --- symmetric bi-linear form. Introduce operators $\mathcal{A}_i$ for $i=1$, \ldots, $N$ by the following relations: $\mathcal{A}_1=\mathcal{S}^g$ and $\mathcal{A}_i = g^{-1}\sigma_{q_i}$ for $i\geq 2$. The dual operators with respect to bi-linear form $g$ are denoted by $(\mathcal{A}_i)_g$.

According to theorem~\ref{th.b-A}, $\mathrm{GL}(E)/\mathbb{Z}_2$-equivalence class of $N$-tuple $(\sigma_{q_1}, \ldots, \sigma_{q_N})$ is uniquely defined by $\mathrm{GL}(E)/\mathbb{Z}_2$-equivalence class of $(N+1)$-tuple $(g, \mathcal{A}_1, \ldots, \mathcal{A}_N)$ or, in another terms, by $\mathrm{O}_g(E)/\mathbb{Z}_2$-equivalence class of $N$-tuple $(\mathcal{A}_1, \ldots, \mathcal{A}_N)$. It follows from theorem~\ref{th.b-A}, that the invariants of such tuples are generated by functions $$J^\sigma_I\colon (\mathrm{\textsf{S}}^kT^*)^{\times N}\to \mathbb{R}, \quad J^\sigma_I:=\tr(\mathcal{X}_{i_1}\ldots \mathcal{X}_{i_s}),$$ where $I=(i_1,\ldots, i_s)$ is multi-index of length $s\leq 2^{\dim E}-1$ and $\mathcal{X}_{i_j}=\mathcal{A}_{i_j}$ or $(\mathcal{A}_{i_j})_g$ (note that $(\mathcal{A}_1)_g = -\mathcal{A}_1$). So regular $N$-tuples of non-degenerated bi-linear forms $(\sigma_{q_1}, \ldots, \sigma_{q_N})$ and $(\widetilde{\sigma}_{q_1}, \ldots, \widetilde{\sigma}_{q_N})$ are $\mathrm{GL}(E)/\mathbb{Z}_2$-equivalent if and only if 
$J_I^{\sigma}=J_I^{\widetilde{\sigma}}$, for all multi-indexes $I$ of length $s\leq 2^{\dim E}-1$. We call functions $J^\sigma_I$ \emph{the Procesi invariants}.

\begin{remark}
In case of real field $\mathbb{R}$ there is one more discrete invariant on $N$-tuple $(\sigma_{q_1}, \ldots, \sigma_{q_N})$, i.e. the signature of symmetric form $g$.
\end{remark}

\section{Classification}

In this section we use the rational $\mathrm{GL}(E)/\mathbb{Z}_2$-invariants $J^\sigma_I$ to classify regular $G$-orbits of the symbols from $\mathbf{Smbl}(\xi,\xi^t)$.

\subsection{General case}

First of all consider the pair of bi-linear forms $(\sigma_{q_i},\sigma_{q_j})$ and corresponding operators $(\mathcal{A}_i, \mathcal{A}_j)$. The $\mathrm{GL}(E)/\mathbb{Z}_2$-stabilizer of this pair of operators is trivial, if 
\begin{quote}
$(*)$ operators $\mathcal{A}_i$ and $\mathcal{A}_j$ have distinct eigenvalues and $[\mathcal{A}_i, \mathcal{A}_j]\neq 0$.     
\end{quote}
Indeed, if some element $A\in \mathrm{GL}(E)/\mathbb{Z}_2$ preserves such operators $\mathcal{A}_i$ and $\mathcal{A}_j$, then this operator is scalar. But $A$ is also orthogonal, so $A=\pm\mathbf{1}$.

Now consider $\mathrm{GL}(E)/\mathbb{Z}_2$-orbit of the $N$-tuple $(\sigma_{q_1}, \ldots, \sigma_{q_N})$ for $N\leq\binom{b+k-1}{k}$, which contains at least two bi-linear forms $\sigma_{q_i}$ and $\sigma_{q_j}$ satisfying the condition $(*)$ for $i$, $j\geq 2$. Then codimension of such orbit equals $Nm^2-m^2=(N-1)m^2$.

\begin{remark}
Note that the condition $(*)$ implies that $N\geq 3$. Case $N=2$ corresponds to the trivial and non-interasting situation $k=1$ and $n=2$.
\end{remark}

Also consider regular $G$-orbit of the symbol $\sigma$. Its codimension equals $\binom{n+k-1}{k} m^2-m^2$.

Let $\mathcal{F}(q_1,\ldots, q_N)$ be a field of rational $\mathrm{GL}(E)/\mathbb{Z}_2$-invariants of $N$-tuples $(\sigma_{q_1}, \ldots, \sigma_{q_N})$, where $N\leq\binom{n+k-1}{k}$. Then according to Rosenlicht theorem, field $\mathcal{F}(q_1,\ldots, q_N)$ is generated by Procesi invariants $J_I^\sigma$ and the transcendence degree of this field equals $$\tr\deg\mathcal{F}(q_1,\ldots,q_N)=(N-1)m^2$$ (see~\cite{Rosenlicht}).

Also denote by $\mathcal{F}$ the field of rational $\mathrm{GL}(E)/\mathbb{Z}_2$-invariants for the action on symbols. Then we have $$\tr\deg \mathcal{F}=\Big(\binom{n+k-1}{k}-1\Big)m^2.$$

We have field extensions $\mathcal{F}\hookrightarrow \mathcal{F}(q_1,\ldots,q_N)$ for all $N\leq\binom{n+k-1}{k}$. If we take $N=\binom{n+k-1}{k}$, then we get the algebraic extension, so the Prochesi invariants $F_I^\sigma$ for $N$-tuples $(\sigma_{q_1},\ldots,\sigma_{q_N})$ also generate the field $\mathcal{F}$ of rational $\mathrm{GL}(E)/\mathbb{Z}_2$ invariants of the symbols.

\begin{theorem}\label{th.E}
The field $\mathcal{F}$ of rational $\mathrm{GL}(E)/\mathbb{Z}_2$-invariants of the symbols $\sigma\in \mathrm{Hom}(E,E^t)\otimes \mathrm{\textsf{S}}^k T^*$ is generated by Procesi invariants $J_I^\sigma$ for $N$-tuples $(\sigma_{q_1}, \ldots, \sigma_{q_N})$, where $N=\binom{n+k-1}{k}$.
\end{theorem}

Let's back to $\mathrm{GL}(T)\times(\mathrm{GL}(E)/\mathbb{Z}_2)$-classification of regular symbols $\sigma\in\mathbf{Smbl}(\xi,\xi^t)$. It follows from theorem~\ref{th.E}, that it is enough to describe the $\mathrm{GL}(T)$-orbits of set $\{J^\sigma_I\}$ of Procesi invariants. Note that we act by linear transformations not only on the arguments $q_1$, \ldots, $q_N$ of the functions $J^\sigma_I$, but on the symbol $\sigma$ too.

First of all let us consider the action of scale transformations on the Procesi invariants. It follows from the formula~(\ref{eq.E}) that this action is trivial, i.e. $$J^{\mu\sigma}_I(\lambda q_1,\ldots,\lambda q_N)=J^\sigma_I(q_1,\ldots,q_N),$$ for all $\lambda$, $\mu\neq 0$. 

So the Procesi invariants are separately homogeneous (in variables $q_i$ and $\sigma$) of degree 0. 

Therefore the the action of group $\mathrm{GL}(T)$ on Procesi invariants in cases $E^t=E^*$ and $E^t=E^\flat$ coincide, i.e. the equivalence of the symbols in all these cases is given by the unique construction.

\begin{remark}
It is clear that the construction of Procesi invariants can be applied also for the case of densities $E^t=\mathrm{Hom}(E,(\Lambda^n)^rT)$, where $r\in \mathbb{Q}$.
\end{remark}

Let us return to the action of the group $\mathrm{GL}(T)$ on Procesi invariants. We have seen, that the numerators and denominators of these invariants are homogeneous polynomials of the same degree. It is well-known (see, for example,~\cite{VinPop}), that the $\mathrm{GL}(T)$-stabilizer of such regular polynomials is trivial, when their degree is more than 2. 

We say, that symbol $\sigma\in \mathbf{Smbl}(\xi,\xi^t)$ is \emph{non-degenerated}, if the condition $(*)$ holds for some pair of points $(q_i,q_j)$ in an $N$-tuple $(q_1,\ldots,q_N)$ in general position and if there exists a Procesi invariant, which numerator degree is more than 2.

Now we are ready to provide the $G$-equivalence of two non-degenerated symbols. We start from the simple case of the first-order operators, i.e. when $k=1$. Then $N=n=\dim T$. 

Take two tuples $(\sigma, q_1,\ldots,q_n)$ and $(\widetilde{\sigma}, \widetilde{q}_1,\ldots,\widetilde{q}_n)$. If $n$-tuples $(q_1,\ldots,q_n)$ and $(\widetilde{q}_1,\ldots,\widetilde{q}_n)$ are in general position, then there exists a unique transformation $A_1\in\mathrm{GL}(T)$, which maps the first tuple to the second tuple. So we may assume that we have two following tuples: $(\sigma, q_1,\ldots,q_n)$ and $(\widetilde{\sigma}, q_1,\ldots,q_n)$. Now assume that $J^\sigma_I(q_1,\ldots,q_n)=J^{\widetilde{\sigma}}_I(q_1,\ldots,q_n)$ for all multi-indexes $I$. Then according to theorem~\ref{th.E} there exists a unique element $A_2\in\mathrm{GL}(E)/\mathbb{Z}_2$, such that $A_2\circ \sigma=\widetilde{\sigma}$. Hence, we get, that symbols $\sigma$ and $\widetilde{\sigma}$ are $G$-equivalent, i.e. $(A_1,A_2)\circ\sigma=\widetilde{\sigma}$. As functions $J^\sigma_I$ are $\mathrm{GL}(E)/\mathbb{Z}_2$-invariant, we get, that these functions separate the $G$-orbits of non-degenerated symbols.

\begin{theorem}
Two non-degenerated symbols $\sigma$, $\widetilde{\sigma}\in\mathbf{Smbl}(\xi,\xi^t)$ are $G$-equivalent if and only if $J^\sigma_I\equiv J^{\widetilde{\sigma}}_I$, for all multi-indexes $I$.
\end{theorem}

The similar result holds for the symbols of differential operators of an arbitrary order. Assume, that $n$, $k\geq 2$. Then $N\geq n+1$.
Note that we cannot take two general tuples $(q_1,\ldots,q_N)$ and $(\widetilde{q}_1,\ldots,\widetilde{q}_N)$ in this case, because such tuples may be non-equivalent with respect to the action of the group $\mathrm{GL}(T)$. But one can take two bases in $T$ and all symmetric monomials $(q_1, \ldots, q_N)$ and $(\widetilde{q}_1, \ldots, \widetilde{q}_N)$ of degree $k$. Such monomials form two bases in  space $\mathrm{\textsf{S}}^kT^*$, and these bases are $\mathrm{GL}(T)$-equivalent. So we can always take the tuples $(\sigma, q_1,\ldots,q_N)$ and $(\widetilde{\sigma}, q_1,\ldots,q_N)$ generated by two bases on $T$ and provide the same constructions as in case $k=1$. We call such tuples \emph{special}.

So, we have the following theorem.

\begin{theorem}
Two non-degenerated symbols $\sigma$, $\widetilde{\sigma}\in\mathbf{Smbl}(\xi,\xi^t)$ are $G$-equivalent if and only if for the special tuple $(q_1,\ldots,q_N)$ in general position $J^\sigma_I(q_1,\ldots,q_N)=J^{\widetilde{\sigma}}_I(q_1,\ldots,q_N)$, for all multi-indexes $I$.
\end{theorem}

\subsection{Self-adjoint case}

In this section we consider the special case of symbols $\sigma\colon \mathrm{\textsf{S}}^kT^*\to \mathrm{Hom}(E,E^t)$. More precisely, we call symbol $\sigma$ \emph{self-adjoint}, if bi-linear form $\sigma_q$ is self-adjoint. 

As in general case, we take $N$-tuple $(\sigma_{q_1},\ldots,\sigma_{q_N})$ of self-adjoint bi-linear forms and associate with each form $\sigma_{q_i}$ linear operator $\mathop{A}_i:=g^{-1}\sigma_{q_i}$, where $i=2$, \ldots, $N$ and $g:=\sigma_{q_1}$.

\begin{remark}
Operator $\mathcal{S}^{\sigma_{q_1}}$ in our case equals zero, so we will not take it into account.
\end{remark}

It follows from theorem~\ref{th.b-A}, that the invariants of our tuples are generated by functions $$K^\sigma_I\colon (\mathrm{\textsf{S}}^kT^*)^{\times N}\to \mathbb{R}, \quad K^\sigma_I:=\tr(\mathcal{X}_{i_1}\ldots \mathcal{X}_{i_s}),$$ where $I=(i_1,\ldots, i_s)$ is multi-index of length $s\leq 2^{\dim E}-1$ and $\mathcal{X}_{i_j}=\mathcal{A}_{i_j}$ or $(\mathcal{A}_{i_j})_g$, where $i_j\geq 2$. 

Now consider $\mathrm{GL}(E)/\mathbb{Z}_2$-orbit of the $N$-tuple $(\sigma_{q_1}, \ldots, \sigma_{q_N})$ of symmetric bi-linear forms, which contains at least two forms $\sigma_{q_i}$ and $\sigma_{q_j}$, $i$, $j\geq 2$ satisfying the condition $(*)$. Then the $\mathrm{GL}(E)/\mathbb{Z}_2$-stabilizer of $N$-tuple $(\sigma_{q_1}, \ldots, \sigma_{q_N})$ is trivial. Then codimension of such orbit equals $N\cdot\frac{m(m+1)}{2}-m^2$.

Also consider regular $G$-orbit of the self-adjoint symbol $\sigma$. Its codimension equals $\binom{n+k-1}{k} \cdot\frac{m(m+1)}{2}-m^2$.

All other constructions from general situation can be applied in our special case without any changes. Summarize these results in the next theorem.

\begin{theorem}\label{th.self-adj}
1. The field $\mathcal{F}$ of rational $\mathrm{GL}(E)/\mathbb{Z}_2$-invariants of the symmetric symbols $\sigma\in \mathrm{Hom}(E,E^t)\otimes \mathrm{\textsf{S}}^k T^*$ is generated by Procesi invariants $K_I^\sigma$ for $N$-tuples $(\sigma_{q_1}, \ldots, \sigma_{q_N})$, where $N=\binom{n+k-1}{k}$.

2. Two non-degenerated self-adjoint symbols $\sigma$, $\widetilde{\sigma}\in\mathbf{Smbl}(\xi,\xi^t)$ are $G$-equivalent if and only if $K^\sigma_I\equiv K^{\widetilde{\sigma}}_I$, for all multi-indexes $I$.
\end{theorem}

\subsection{Skew-symmetric case}

Finally, we consider another special case of symbols $\sigma\colon \mathrm{\textsf{S}}^kT^*\to \mathrm{Hom}(E,E^t)$. Namely, we call symbol $\sigma$ \emph{skew-symmetric}, if bi-linear form $\sigma_q$ is skew-symmetric. 

As in general case, we assume that $\dim E=m=2r$ is even and take $N$-tuple $(\sigma_{q_1},\ldots,\sigma_{q_N})$ of skew-symmetric non-degenerated bi-linear forms and associate with each form $\sigma_{q_i}$ linear operator $\mathop{A}_i:=\omega^{-1}\sigma_{q_i}$, where $i=2$, \ldots, $N$ and $\omega:=\sigma_{q_1}$.

It follows from theorem~\ref{th.b-A}, applying to the symplectic group $\mathrm{Sp}_\omega(E)$ (see~\cite{Pr1}), that the invariants of our tuples are generated by functions $$L^\sigma_I\colon (\mathrm{\textsf{S}}^kT^*)^{\times N}\to \mathbb{R}, \quad K^\sigma_I:=\tr(\mathcal{X}_{i_1}\ldots \mathcal{X}_{i_s}),$$ where $I=(i_1,\ldots, i_s)$ is multi-index of length $s\leq 2^{\dim E}-1$ and $\mathcal{X}_{i_j}=\mathcal{A}_{i_j}$ or $(\mathcal{A}_{i_j})_\omega$, where $i_j\geq 2$. 

Now consider $\mathrm{GL}(E)/\mathbb{Z}_2$-orbit of the $N$-tuple $(\sigma_{q_1}, \ldots, \sigma_{q_N})$ of skew-symmetric bi-linear forms, which contains at least two forms $\sigma_{q_i}$ and $\sigma_{q_j}$, $i$, $j\geq 2$ satisfying the condition $(*)$. Let us prove that the $\mathrm{GL}(E)/\mathbb{Z}_2$-stabilizer of pair $(\sigma_{q_i}, \sigma_{q_j})$ is trivial. 

Put $\alpha:=\sigma_{q_i}$ and $\beta:=\sigma_{q_j}$ ---  skew-symmetric bi-linear forms. Denote $\mathcal{A}:=\mathcal{A}_i$ and $\mathcal{B}:=\mathcal{B}_j$, i.e. $\alpha(x,y)=\omega(\mathcal{A}x,y)$ and $\beta(x,y)=\omega(\mathcal{B}x,y)$ for all $x$, $y\in E$. Then $$\omega(\mathcal{A}x,y)=\alpha(x,y)=-\alpha(y,x)=-\omega(\mathcal{A}y,x)=\omega(x,\mathcal{A}y),$$ i.e. $\mathcal{A}=\mathcal{A}_\omega$.

Recall, that Pfiffian $\mathrm{Pf}(\alpha)$ of the form $\alpha$ is defined by relation $\alpha^{\wedge r}=\mathrm{Pf}(\alpha)\omega^{\wedge r}$.

With the help of Pfaffian one can calculate eigenvalues of the skew-symmetric forms. Namely, consider the equation $\mathrm{Pf}(\omega-\lambda \alpha)=0$. Assume that this equation has $r$ distinct roots $\lambda_1$, \ldots, $\lambda_r$. The corresponding eigenspaces $E_1$, \ldots, $E_r$ have dimension 2 and decompose space $E$ in the direct sum $E=\bigoplus\limits_{i=1}^r E_i$, and $A\mid_{E_i}=\lambda_i\cdot \mathbf{1}$.

Note that if $e_i\in E_i$ and $e_j\in E_j$ are two arbitrary eigenvectors for $i\neq j$, then $$\lambda_i\omega(e_i,e_j) = \omega(\mathcal{A}e_i,e_j)=\omega(e_1,\mathcal{A}e_j)=\lambda_j\omega(e_i,e_j),$$ hence, $\omega(e_i,e_j)=0$ for all $e_i\in E_i$ and $e_j\in E_j$.

Now let $\mathcal{C}\in\mathrm{Sp}_\omega(E)$ be non-trivial operator, which commutes with operator $\mathcal{A}$. Then let $\mathcal{C}_i:=\mathcal{C}\mid_{E_i}$ and $\omega_i:=\omega\mid_{E_i}$ be the restrictions of operator $\mathcal{C}$ and skew-symmetric form $\omega$ on space $E_i$. Hence $\mathcal{C}_i\in \mathrm{Sp}_{\omega_i}(E_i)$ and there exists an eigenbasis $\{e_i, f_i\}$ of operator $\mathcal{C}_i$ on space $E_i$.

So we construct the common eigenbasis $\{e_1,f_1,\ldots,e_r,f_r\}$ for operators $\mathcal{C}$ and $\mathcal{A}$. In the same way one can construct the common eigenbasis for operators $\mathcal{C}$ and $\mathcal{B}$. Finally, we get, that operators $\mathcal{A}$ and $\mathcal{B}$ have common eigenbasis, which is impossible, because $[\mathcal{A},\mathcal{B}]\neq 0$.

Now consider $\mathrm{GL}(E)/\mathbb{Z}_2$-orbit of the $N$-tuple $(\sigma_{q_1}, \ldots, \sigma_{q_N})$ of skew-symmetric bi-linear forms, which contains at least two forms $\sigma_{q_i}$ and $\sigma_{q_j}$, $i$, $j\geq 2$ satisfying the condition $(*)$. It follows, that codimension of such orbit equals $N\cdot\frac{m(m-1)}{2}-m^2$.

Also consider regular $G$-orbit of the skew-symmetric symbol $\sigma$. Its codimension equals $\binom{n+k-1}{k} \cdot\frac{m(m-1)}{2}-m^2$.

All other constructions from general situation can be applied in our special case without any changes. Summarize these results in the next theorem.

\begin{theorem}
1. The field $\mathcal{F}$ of rational $\mathrm{GL}(E)/\mathbb{Z}_2$-invariants of the skew-symmetric symbols $\sigma\in \mathrm{Hom}(E,E^t)\otimes \mathrm{\textsf{S}}^k T^*$ is generated by Procesi invariants $L_I^\sigma$ for $N$-tuples $(\sigma_{q_1}, \ldots, \sigma_{q_N})$, where $N=\binom{n+k-1}{k}$.

2. Two non-degenerated skew-symmetric symbols $\sigma$, $\widetilde{\sigma}\in\mathbf{Smbl}(\xi,\xi^t)$ are $G$-equivalent if and only if $L^\sigma_I\equiv L^{\widetilde{\sigma}}_I$ for all multi-indexes $I$.
\end{theorem}


\end{document}